\documentclass{amsart}
\usepackage{amsmath}
\usepackage{amsfonts}

\newtheorem{theorem}{Theorem}
\theoremstyle{plain}

\newtheorem{conclusion}{Conclusion}

\newtheorem{corollary}{Corollary}

\newtheorem{definition}{Definition}

\newtheorem{lemma}{Lemma}

\newtheorem{remark}{Remark}

\numberwithin{equation}{section}
\input{tcilatex}

\begin{document}
\title[Hermite--Hadamard type inequalities via differentiable\textbf{...}]{%
Hermite--Hadamard type inequalities via differentiable $h_{\varphi }-$%
preinvex functions for Fractional Integrals }
\author{Abdullah AKKURT}
\address{[Department of Mathematics, Faculty of Science and Arts, University
of Kahramanmara\c{s} S\"{u}t\c{c}\"{u} \.{I}mam, 46000, Kahramanmara\c{s},
Turkey}
\email{{\small abdullahmat@gmail.com}}
\author{H\"{u}seyin YILDIRIM}
\address{[Department of Mathematics, Faculty of Science and Arts, University
of Kahramanmara\c{s} S\"{u}t\c{c}\"{u} \.{I}mam, 46000, Kahramanmara\c{s},
Turkey}
\email{hyildir@ksu.edu.tr}
\keywords{\textbf{}integral inequalities, fractional integrals,
Hermite-Hadamard Inequality, Preinvex functions, H\"{o}lder's inequality.}

\begin{abstract}
In this paper, we consider a new class of convex functions which is called $%
h_{\varphi }-$preinvex functions. We prove several Hermite--Hadamard type
inequalities for differentiable $h_{\varphi }-$preinvex functions via
Fractional Integrals. Some special cases are also discussed. Our results
extend and improve the corresponding ones in the literature.
\end{abstract}

\maketitle

\section{Introduction}

The following inequality is well-known in the literature as Hermite-Hadamard
inequality. Let $f:I\subseteq 
\mathbb{R}
\rightarrow 
\mathbb{R}
$ be a convex function with $a<b$ and $a,b\in I$. Then the following holds%
\begin{equation}
\begin{array}{c}
f(\dfrac{a+b}{2})\leq \dfrac{1}{b-a}\dint\limits_{a}^{b}f(x)dx\leq \dfrac{%
f(a)+f(b)}{2}.%
\end{array}
\tag{1.1}
\end{equation}

Recently, Hermite-Hadamard type inequality has been the subject of intensive
research. Hermite--Hadamard inequality can be considered as necessary and
sufficient condition for a function to be convex. It provides estimates of
the mean value of continuous convex function.

In recent years, several extensions and generalizations have been proposed
for classical convexity (see $[1-24]$). A significant generalization of
classical convex functions is that of $\varphi -$convex functions introduced
by Noor $[10]$. Noor $[10]$ has investigated the basic properties of $%
\varphi -$convex functions and showed that $\varphi -$convex functions are
nonconvex functions. Noor $[9]$ extended Hermite--Hadamard type inequalities
for $\varphi -$convex functions.

Motivated by the ongoing research on generalizations and extensions of
classical convexity, Varosanec $[23]$ introduced the class of $h-$convex
functions. Noor et al. $[13]$ introduced another generalization of classical
convexity which is called as $h_{\varphi }-$convex functions. Finally Noor
et al. Several Hermite--Hadamard type inequalities are proved for $%
h_{\varphi }-$preinvex functions$\left[ 20\right] $.

Weir and Mond $[24]$ investigated the class of preinvex functions. It is
well-known that the preinvex functions and invex sets may not be convex
functions and convex sets. For recent investigation on preinvexity (see $%
[1,4,5-8,17,22]$).

Motivated and inspired by the recent research in this field, we consider a
new class of convex functions, which is called $h_{\varphi }-$preinvex
functions. This class includes several known and new classes of convex
functions such as $\varphi -$convex functions $[10]$, preinvex functions $%
[21]$, $h-$convex functions $[23]$ as special cases. We derive several new
Hermite--Hadamard type\ fractional integral inequalities for $h_{\varphi }-$%
preinvex functions and their variant forms. Results obtained in this paper
continue to hold for these special cases. Our results represent significant
generalized of the previous results. The interested readers are encouraged
to find novel and innovative applications of $h_{\varphi }-$preinvex
functions and fractional integrals.

\section{Preliminares}

Let $%
\mathbb{R}
^{n}$ be the finite dimensional Euclidian space. Also let $0\leq \varphi
\leq \dfrac{\pi }{2}$ be a continuous function.

\begin{definition}
$([12])$ The set $K_{\varphi n}$ in $%
\mathbb{R}
^{n}$ is said to be $\varphi -$invex at $u$ with respect to $\varphi (.)$,
if there exists a bifunction $\eta (.,.):K_{\varphi n}\times K_{\varphi
n}\rightarrow 
\mathbb{R}
$, such that 
\begin{equation}
\begin{array}{c}
u+te^{i\varphi }\eta (v,u)\in K_{\varphi n},\ \ \ \ \ \forall u,v\in
K_{\varphi n},\ t\in \lbrack 0,1].%
\end{array}
\tag{2.1}
\end{equation}%
The $\varphi -$invex set $K_{\varphi n}$ is also called $\varphi n$%
-connected set. Note that the convex set with $\varphi =0$ and $\eta
(v,u)=v-u$ is a $\varphi -$invex set, but the converse is not true.
\end{definition}

\begin{definition}
$([10])$ Let $K_{\varphi }$ be a set in $%
\mathbb{R}
^{n}$. Then the set $K_{\varphi }$ is said to be $\varphi -$convex with
respect to $\varphi $, if and only if 
\begin{equation*}
\begin{array}{c}
u+te^{i\varphi }(v,u)\in K_{\varphi },\ \ \ \ \ \forall u,v\in K_{\varphi
},\ t\in \lbrack 0,1].%
\end{array}%
\end{equation*}%
For $\varphi =0$, the set $K_{\varphi }$ reduces to the classical convex set 
$K$. That is,%
\begin{equation*}
\begin{array}{c}
u+t(v,u)\in K,\ \ \ \ \ \forall u,v\in K,\ t\in \lbrack 0,1].%
\end{array}%
\end{equation*}
\end{definition}

\begin{definition}
$([24])$ A set $K_{n}$ is said to be invex set with respect to bifunction $%
\eta (.,.)$, if%
\begin{equation*}
\begin{array}{c}
u+t\eta (v,u)\in K_{n},\ \ \ \ \ \forall u,v\in K_{n},\ t\in \lbrack 0,1].%
\end{array}%
\end{equation*}%
The invex set $K_{n}$ is also called $\eta -$connected set.
\end{definition}

\begin{definition}
$([20])$Let $h:J\subseteq 
\mathbb{R}
\rightarrow 
\mathbb{R}
$ be a nonnegative function. A function f on the set $K_{\varphi n}$ is said
to be $h_{\varphi }-$preinvex function with respect to $\varphi $ and
bifunction $\eta $, if%
\begin{equation}
\begin{array}{c}
f(u+te^{i\varphi }\eta (v,u))\leq h(1-t)f(u)+h(t)f(v),\ \ \ \ \forall u,v\in
K_{\varphi n},\ t\in \lbrack 0,1].%
\end{array}
\tag{2.2}
\end{equation}
\end{definition}

\begin{remark}
One can deduce several known concepts from Definition $4$ as:
\end{remark}

$(1)$ For $h(t)=t$ Definition $4$ reduces to the definition for $\varphi -$%
preinvex functions (see $[12]$).

$(2)$ For $\varphi =0$ Definition $4$ reduces to the definition for $h-$%
preinvex functions (see $[17]$).

$(3)$ For $\eta (v,u)=v-u$ Definition $4$ reduces to the definition for $%
h_{\varphi }-$convex functions (see $[13]$).

$(4)$ For $\varphi =0$ and $\eta (v,u)=v-u$ Definition $4$ reduces to the
definition for $h-$convex functions (see $[23]$).

Now we discuss some special cases of Definition $4$.

$\mathbf{I.}$ For $h(t)=t^{s}$ where $s\in (0,1]$ in $(2.2)$ we have the
definition for $s_{\varphi }-$preinvex functions.

\begin{definition}
A function f on the set $K_{\varphi n}$ is said to be $s_{\varphi }-$%
preinvex function with respect to $\varphi $ and $\eta $, if%
\begin{equation}
\begin{array}{c}
f(u+te^{i\varphi }\eta (v,u))\leq (1-t)^{s}f(u)+t^{s}f(v),\ \ \ \ \forall
u,v\in K_{\varphi n},\ t\in \lbrack 0,1].%
\end{array}
\tag{2.3}
\end{equation}
\end{definition}

$\mathbf{II.}$ For $h(t)=1$ in $(2.2)$ we have the definition for $%
P_{\varphi }-$preinvex functions.

\begin{definition}
A function $f$ on the set $K_{\varphi n}$ is said to be $s_{\varphi }-$%
preinvex function with respect to $\varphi $ and $\eta $, if 
\begin{equation}
\begin{array}{c}
f(u+te^{i\varphi }\eta (v,u))\leq f(u)+f(v),\ \ \ \ \forall u,v\in
K_{\varphi n},\ t\in \lbrack 0,1].%
\end{array}
\tag{2.4}
\end{equation}
\end{definition}

\begin{definition}
$\left( \left[ 25\right] \right) \ $Let $f\in L^{1}[a,b]$. The
Riemann-Liouville fractional integral $J_{a^{+}}^{\alpha }f\left( x\right) $
and $J_{b^{-}}^{\alpha }f\left( x\right) $ of order $\alpha >0$ are defined
by%
\begin{equation}
\begin{array}{cc}
J_{a^{+}}^{\alpha }\left[ f\left( x\right) \right] =\frac{1}{\Gamma \left(
\alpha \right) }\dint\limits_{a}^{x}\left( x-t\right) ^{\alpha -1}f\left(
t\right) dt & x>a%
\end{array}
\tag{2.5}
\end{equation}%
and%
\begin{equation}
\begin{array}{cc}
J_{b^{-}}^{\alpha }\left[ f\left( x\right) \right] =\frac{1}{\Gamma \left(
\alpha \right) }\dint\limits_{x}^{b}(t-x)^{\alpha -1}f\left( t\right) dt & 
x<b%
\end{array}
\tag{2.6}
\end{equation}%
respectively, where $\Gamma \left( \alpha \right) =\dint\limits_{0}^{\infty
}e^{-u}u^{\alpha -1}du$ is Gamma function and $%
J_{a^{+}}^{0}f(x)=J_{b^{-}}^{0}f(x)=f(x).$
\end{definition}

\section{Main results}

In this section, we will discuss our main results for $h_{\varphi }-$%
preinvex functions and fractional integrals.

Using the technique of $[1]$ and $[20]$, we prove the following Lemma which
play a key role in our study.

\begin{lemma}
Let $I\subseteq 
\mathbb{R}
$ be a open invex set with respect to bifunction $\eta :I\times I\rightarrow 
\mathbb{R}
$ where $\eta (b,a)>0$. If $f^{^{\prime }}\in L_{1}[a,a+e^{i\varphi }\eta
(b,a)]\ $and $\alpha \geq 0$, then%
\begin{equation}
\begin{array}{l}
\left[ f(a)+f(a+e^{i\varphi }\eta (b,a))\right] -\dfrac{\Gamma (\alpha +1)}{%
\left[ e^{i\varphi }\eta (b,a)\right] ^{\alpha }}\left\{ J_{a^{+}}^{\alpha
}f(x)+J_{\left( a+e^{i\varphi }\eta (b,a)\right) ^{-}}^{\alpha }f(x)\right\}
\\ 
=e^{i\varphi }\eta (b,a)\dint\limits_{0}^{1}\left[ t^{\alpha }-\left(
1-t\right) ^{\alpha }\right] f^{^{\prime }}(a+te^{i\varphi }\eta (b,a))dt.%
\end{array}
\tag{3.1}
\end{equation}
\end{lemma}

\begin{proof}
Let%
\begin{equation*}
\begin{array}{l}
\dint\limits_{0}^{1}\left[ \left( 1-t\right) ^{\alpha }-t^{\alpha }\right]
f^{^{\prime }}(a+te^{i\varphi }\eta (b,a))dt=\left. \dfrac{\left( \left(
1-t\right) ^{\alpha }-t^{\alpha }\right) f(a+te^{i\varphi }\eta (b,a))}{%
e^{i\varphi }\eta (b,a)}\right\vert _{0}^{1} \\ 
+\dfrac{\alpha }{e^{i\varphi }\eta (b,a)}\dint\limits_{0}^{1}\left[ \left(
1-t\right) ^{\alpha -1}+t^{\alpha -1}\right] f(a+te^{i\varphi }\eta (b,a))dt
\\ 
=-\dfrac{\left[ f(a)+f(a+e^{i\varphi }\eta (b,a))\right] }{e^{i\varphi }\eta
(b,a)}+\dfrac{\alpha }{e^{i\varphi }\eta (b,a)}\dint\limits_{0}^{1}\left[
\left( 1-t\right) ^{\alpha -1}+t^{\alpha -1}\right] f(a+te^{i\varphi }\eta
(b,a))dt \\ 
=-\frac{1}{e^{i\varphi }\eta (b,a)}\left[ \left[ f(a)+f(a+e^{i\varphi }\eta
(b,a))\right] \right. \\ 
\ \ \ \ \ \ \ \ \ \ \ \ \ \ \ \ -\left. \alpha
\dint\limits_{a}^{a+e^{i\varphi }\eta (b,a)}\left[ \left( 1-\frac{u-a}{%
e^{i\varphi }\eta (b,a)}\right) ^{\alpha -1}+\left( \frac{u-a}{e^{i\varphi
}\eta (b,a)}\right) ^{\alpha -1}\right] f(u)\frac{du}{e^{i\varphi }\eta (b,a)%
}\right] \\ 
=-\dfrac{\left[ f(a)+f(a+e^{i\varphi }\eta (b,a))\right] }{e^{i\varphi }\eta
(b,a)} \\ 
\ \ \ \ \ \ \ \ \ \ \ \ \ \ \ +\dfrac{\alpha }{e^{i\varphi }\eta (b,a)}%
\dint\limits_{a}^{a+e^{i\varphi }\eta (b,a)}\left[ \left( a+e^{i\varphi
}\eta (b,a)-u\right) ^{\alpha -1}+\left( \frac{u-a}{e^{i\varphi }\eta (b,a)}%
\right) ^{\alpha -1}\right] f\left( u\right) du \\ 
=-\dfrac{\left[ f(a)+f(a+e^{i\varphi }\eta (b,a))\right] }{e^{i\varphi }\eta
(b,a)}+\dfrac{\Gamma \left( \alpha +1\right) }{\left[ e^{i\varphi }\eta (b,a)%
\right] ^{\alpha +1}}\left[ J_{a^{+}}^{\alpha }f(x)+J_{\left( a+e^{i\varphi
}\eta (b,a)\right) ^{-}}^{\alpha }f(x)\right] .%
\end{array}%
\end{equation*}%
This complete the proof.
\end{proof}

\begin{theorem}
Let $I\subseteq 
\mathbb{R}
$ be a open invex set with respect to bifunction $\eta :I\times I\rightarrow 
\mathbb{R}
.$Suppose $f:I\rightarrow 
\mathbb{R}
$ is a differentiable function such that $f^{^{\prime }}\in
L_{1}[a,a+e^{i\varphi }\eta (b,a)].\ $Let $\left\vert f^{^{\prime
}}\right\vert $be $h_{\varphi }-$preinvex function. Then, for $\eta (b,a)>0\ 
$and $\alpha >0,$%
\begin{equation}
\begin{array}{l}
\left\vert f(a)+f(a+e^{i\varphi }\eta (b,a))-\dfrac{\Gamma (\alpha +1)}{%
\left[ e^{i\varphi }\eta (b,a)\right] ^{\alpha }}\left[ J_{a^{+}}^{\alpha
}f(x)+J_{\left( a+e^{i\varphi }\eta (b,a)\right) ^{-}}^{\alpha }f(x)\right]
\right\vert \\ 
\leq e^{i\varphi }\eta (b,a)\left[ \left\vert f^{^{\prime }}(a)\right\vert
+\left\vert f^{^{\prime }}(b)\right\vert \right] \dint\limits_{0}^{1}\left%
\vert \left( 1-t\right) ^{\alpha }-t^{\alpha }\right\vert h(t)dt.%
\end{array}
\tag{3.2}
\end{equation}
\end{theorem}

\begin{proof}
Using Lemma 2 and $h_{\varphi }-$preinvexity of $\left\vert f^{^{\prime
}}\right\vert $, we have%
\begin{equation*}
\begin{array}{l}
\left\vert f(a)+f(a+e^{i\varphi }\eta (b,a))-\dfrac{\Gamma (\alpha +1)}{%
\left[ e^{i\varphi }\eta (b,a)\right] ^{\alpha }}\left[ J_{a^{+}}^{\alpha
}f(x)+J_{\left( a+e^{i\varphi }\eta (b,a)\right) ^{-}}^{\alpha }f(x)\right]
\right\vert \\ 
\leq e^{i\varphi }\eta (b,a)\dint\limits_{0}^{1}\left\vert \left( 1-t\right)
^{\alpha }-t^{\alpha }\right\vert \left\vert f^{^{\prime }}(a+te^{i\varphi
}\eta (b,a))\right\vert dt \\ 
\leq e^{i\varphi }\eta (b,a)\dint\limits_{0}^{1}\left\vert \left( 1-t\right)
^{\alpha }-t^{\alpha }\right\vert \left( h\left( 1-t\right) \left\vert
f^{^{\prime }}(a)\right\vert +h\left( t\right) \left\vert f^{^{\prime
}}(b)\right\vert \right) dt%
\end{array}%
\end{equation*}%
\begin{equation*}
\begin{array}{l}
=e^{i\varphi }\eta (b,a)\left[ \dint\limits_{0}^{1}\left\vert \left(
1-t\right) ^{\alpha }-t^{\alpha }\right\vert h\left( 1-t\right) \left\vert
f^{^{\prime }}(a)\right\vert dt\right. \\ 
\ \ \ \ \ \ \ \ \ \ \ \ \ \ \ \ \ \left. +\dint\limits_{0}^{1}\left\vert
\left( 1-t\right) ^{\alpha }-t^{\alpha }\right\vert h\left( t\right)
\left\vert f^{^{\prime }}(a)\right\vert dt\right] \\ 
=e^{i\varphi }\eta (b,a)\left[ \dint\limits_{0}^{1}\left\vert t^{\alpha
}-\left( 1-t\right) ^{\alpha }\right\vert h\left( t\right) \left\vert
f^{^{\prime }}(a)\right\vert dt\right. \\ 
\ \ \ \ \ \ \ \ \ \ \ \ \ \ \ \ \ \left. +\dint\limits_{0}^{1}\left\vert
\left( 1-t\right) ^{\alpha }-t^{\alpha }\right\vert h\left( t\right)
\left\vert f^{^{\prime }}(b)\right\vert dt\right] \\ 
=e^{i\varphi }\eta (b,a)\left[ \left\vert f^{^{\prime }}(a)\right\vert
+\left\vert f^{^{\prime }}(b)\right\vert \right] \dint\limits_{0}^{1}\left%
\vert \left( 1-t\right) ^{\alpha }-t^{\alpha }\right\vert h(t)dt.%
\end{array}%
\end{equation*}%
This completes the proof.
\end{proof}

Now we have some special cases.

\textbf{I.} If $h(t)=t$, then we have the result for $\varphi -$preinvexity.

\begin{corollary}
Let $I\subseteq 
\mathbb{R}
$ be a open invex set with respect to $\eta :I\times I\rightarrow 
\mathbb{R}
.\ $Suppose that $f:I\rightarrow 
\mathbb{R}
$ is a differentiable function such that $f^{^{\prime }}\in
L_{1}[a,a+e^{i\varphi }\eta (b,a)].\ $If $\left\vert f^{^{\prime
}}\right\vert $is $\varphi -$preinvex on $I,\ $then, for $\eta (b,a)>0,$%
\begin{equation}
\begin{array}{l}
\left\vert f(a)+f(a+e^{i\varphi }\eta (b,a))-\dfrac{\Gamma (\alpha +1)}{%
\left[ e^{i\varphi }\eta (b,a)\right] ^{\alpha }}\left[ J_{a^{+}}^{\alpha
}f(x)+J_{\left( a+e^{i\varphi }\eta (b,a)\right) ^{-}}^{\alpha }f(x)\right]
\right\vert \\ 
\leq e^{i\varphi }\eta (b,a)\left[ \left\vert f^{^{\prime }}(a)\right\vert
+\left\vert f^{^{\prime }}(b)\right\vert \right] \dfrac{1}{\alpha +2}\left[
1-\dfrac{1}{2^{\alpha +2}}\right] .%
\end{array}
\tag{3.3}
\end{equation}
\end{corollary}

\textbf{II. }If $h(t)=t^{s}$, then we have the result for $s_{\varphi }-$%
preinvexity

\begin{corollary}
Let $I\subseteq 
\mathbb{R}
$ be a open invex set with respect to $\eta :I\times I\rightarrow 
\mathbb{R}
$. Suppose that $f:I\rightarrow 
\mathbb{R}
$ is a differentiable function such that $f^{^{\prime }}\in
L_{1}[a,a+e^{i\varphi }\eta (b,a)]$. If $\left\vert f^{^{\prime
}}\right\vert $ is $s_{\varphi }-$preinvex on $I$, then, for $\eta (b,a)>0$
and $s\in (0,1]$,%
\begin{equation}
\begin{array}{c}
\left\vert f(a)+f(a+e^{i\varphi }\eta (b,a))-\dfrac{\Gamma (\alpha +1)}{%
\left[ e^{i\varphi }\eta (b,a)\right] ^{\alpha }}\left[ J_{a^{+}}^{\alpha
}f(x)+J_{\left( a+e^{i\varphi }\eta (b,a)\right) ^{-}}^{\alpha }f(x)\right]
\right\vert \\ 
\leq e^{i\varphi }\eta (b,a)\left[ \left\vert f^{^{\prime }}(a)\right\vert
+\left\vert f^{^{\prime }}(b)\right\vert \right] \left[ B_{\frac{1}{2}%
}\left( \alpha +1,s+1\right) \right. \\ 
\left. \left( -B_{\frac{1}{2}}\left( s+1,\alpha +1\right) +\dfrac{1}{\alpha
+s+1}\left[ 1-\dfrac{1}{2^{s+\alpha }}\right] \right) \right] .%
\end{array}
\tag{3.4}
\end{equation}
\end{corollary}

\textbf{III. }If $h(t)=1$, then we have the result for $P_{\varphi }-$%
preinvexity.

\begin{corollary}
Let $I\subseteq 
\mathbb{R}
$ be a open invex set with respect to $\eta :I\times I\rightarrow 
\mathbb{R}
$. Suppose that $f:I\rightarrow 
\mathbb{R}
$ is a differentiable function such that $f^{^{\prime }}\in
L_{1}[a,a+e^{i\varphi }\eta (b,a)]$. If $\left\vert f^{^{\prime
}}\right\vert $ is $P_{\varphi }-$preinvex on $I$, then, for $\eta (b,a)>0$,%
\begin{equation}
\begin{array}{l}
\left\vert f(a)+f(a+e^{i\varphi }\eta (b,a))-\dfrac{\Gamma (\alpha +1)}{%
\left[ e^{i\varphi }\eta (b,a)\right] ^{\alpha }}\left[ J_{a^{+}}^{\alpha
}f(x)+J_{\left( a+e^{i\varphi }\eta (b,a)\right) ^{-}}^{\alpha }f(x)\right]
\right\vert \\ 
\leq e^{i\varphi }\eta (b,a)\left[ \left\vert f^{^{\prime }}(a)\right\vert
+\left\vert f^{^{\prime }}(b)\right\vert \right] \dfrac{2}{\alpha +1}\left[
1-\dfrac{1}{2^{\alpha }}\right] .%
\end{array}
\tag{3.5}
\end{equation}
\end{corollary}

\begin{theorem}
Let $I\subseteq 
\mathbb{R}
$ be a open invex set with respect to $\eta :I\times I\rightarrow 
\mathbb{R}
$. Suppose that $f:I\rightarrow 
\mathbb{R}
$ is a differentiable function such that $f^{^{\prime }}\in
L_{1}[a,a+e^{i\varphi }\eta (b,a)]$. Let $\left\vert f^{^{\prime
}}\right\vert ^{p}$ be $h_{\varphi }-$preinvex on $I$, $\frac{1}{p}+\frac{1}{%
q}=1\ $and $\eta (b,a)>0$. Then%
\begin{equation}
\begin{array}{l}
\left\vert f(a)+f(a+e^{i\varphi }\eta (b,a))-\dfrac{\Gamma (\alpha +1)}{%
\left[ e^{i\varphi }\eta (b,a)\right] ^{\alpha }}\left[ J_{a^{+}}^{\alpha
}f(x)+J_{\left( a+e^{i\varphi }\eta (b,a)\right) ^{-}}^{\alpha }f(x)\right]
\right\vert \\ 
\leq \left( \dfrac{2}{\alpha p+1}\right) ^{\frac{1}{p}}e^{i\varphi }\eta
(b,a)\left[ 1-\dfrac{1}{2^{\alpha p}}\right] ^{\frac{1}{p}}\left\{ \left[
\left\vert f^{^{\prime }}(a)\right\vert ^{\frac{p}{p-1}}+\left\vert
f^{^{\prime }}(b)\right\vert ^{\frac{p}{p-1}}\right] \dint%
\limits_{0}^{1}h(t)dt\right\} ^{\frac{p-1}{p}}.%
\end{array}
\tag{3.6}
\end{equation}
\end{theorem}

\begin{proof}
Using Lemma 2, we have%
\begin{equation*}
\begin{array}{l}
\left\vert f(a)+f(a+e^{i\varphi }\eta (b,a))-\dfrac{\Gamma (\alpha +1)}{%
\left[ e^{i\varphi }\eta (b,a)\right] ^{\alpha }}\left[ J_{a^{+}}^{\alpha
}f(x)+J_{\left( a+e^{i\varphi }\eta (b,a)\right) ^{-}}^{\alpha }f(x)\right]
\right\vert \\ 
=\left\vert e^{i\varphi }\eta (b,a)\dint\limits_{0}^{1}\left[ \left(
1-t\right) ^{\alpha }-t^{\alpha }\right] f^{^{\prime }}(a+te^{i\varphi }\eta
(b,a))dt\right\vert \\ 
\leq e^{i\varphi }\eta (b,a)\dint\limits_{0}^{1}\left\vert \left( 1-t\right)
^{\alpha }-t^{\alpha }\right\vert \left\vert f^{^{\prime }}(a+te^{i\varphi
}\eta (b,a))\right\vert dt.%
\end{array}%
\end{equation*}%
From Holders inequality, we have%
\begin{equation*}
\begin{array}{l}
\left\vert f(a)+f(a+e^{i\varphi }\eta (b,a))-\dfrac{\Gamma (\alpha +1)}{%
\left[ e^{i\varphi }\eta (b,a)\right] ^{\alpha }}\left[ J_{a^{+}}^{\alpha
}f(x)+J_{\left( a+e^{i\varphi }\eta (b,a)\right) ^{-}}^{\alpha }f(x)\right]
\right\vert \\ 
\leq e^{i\varphi }\eta (b,a)\left( \dint\limits_{0}^{1}\left\vert \left(
1-t\right) ^{\alpha }-t^{\alpha }\right\vert ^{p}dt\right) ^{\frac{1}{p}%
}\left( \dint\limits_{0}^{1}\left\vert f^{^{\prime }}(a+te^{i\varphi }\eta
(b,a))\right\vert ^{q}dt\right) ^{\frac{1}{q}} \\ 
=e^{i\varphi }\eta (b,a)\left( \dint\limits_{0}^{\frac{1}{2}}\left( \left(
1-t\right) ^{\alpha }-t^{\alpha }\right) ^{p}dt+\dint\limits_{\frac{1}{2}%
}^{1}\left( t^{\alpha }-\left( 1-t\right) ^{\alpha }\right) ^{p}dt\right) ^{%
\frac{1}{p}} \\ 
\ \ \ \ \ \ \ \ \ \ \ \ \ \ \ \ \times \left( \dint\limits_{0}^{1}\left\vert
f^{^{\prime }}(a+te^{i\varphi }\eta (b,a))\right\vert ^{q}dt\right) ^{\frac{1%
}{q}}.%
\end{array}%
\end{equation*}%
Here $h_{\varphi }-$preinvexity of $\left\vert f^{^{\prime }}\right\vert
^{p} $, we have following inequality%
\begin{equation*}
\begin{array}{l}
\left\vert f(a)+f(a+e^{i\varphi }\eta (b,a))-\dfrac{\Gamma (\alpha +1)}{%
\left[ e^{i\varphi }\eta (b,a)\right] ^{\alpha }}\left[ J_{a^{+}}^{\alpha
}f(x)+J_{\left( a+e^{i\varphi }\eta (b,a)\right) ^{-}}^{\alpha }f(x)\right]
\right\vert \\ 
\leq \dfrac{2^{\frac{1}{p}}}{\left( \alpha p+1\right) ^{\frac{1}{p}}}%
e^{i\varphi }\eta (b,a)\left[ 1-\dfrac{1}{2^{\alpha p}}\right] ^{\dfrac{1}{p}%
}\left( \dint\limits_{0}^{1}\left[ h(1-t)\left\vert f^{^{\prime
}}(a)\right\vert ^{q}+h(t)\left\vert f^{^{\prime }}(b)\right\vert ^{q}\right]
dt\right) ^{\frac{1}{q}} \\ 
=\left( \dfrac{2}{\alpha p+1}\right) ^{\frac{1}{p}}e^{i\varphi }\eta (b,a)%
\left[ 1-\dfrac{1}{2^{\alpha p}}\right] ^{\dfrac{1}{p}}\left\{ \left[
\left\vert f^{^{\prime }}(a)\right\vert ^{\frac{p}{p-1}}+\left\vert
f^{^{\prime }}(b)\right\vert ^{\frac{p}{p-1}}\right] \dint%
\limits_{0}^{1}h(t)dt\right\} ^{\frac{p-1}{p}}.%
\end{array}%
\end{equation*}%
This completes the proof.
\end{proof}

Now we have some special cases for $(3.6)$.

\textbf{I.} If $h(t)=t$, then we have the result for $\varphi -$preinvexity.

\begin{corollary}
Let $I\subseteq 
\mathbb{R}
$ be a open invex set with respect to $\eta :I\times I\rightarrow 
\mathbb{R}
.\ $Suppose that $f:I\rightarrow 
\mathbb{R}
$ is a differentiable function such that $f^{^{\prime }}\in
L_{1}[a,a+e^{i\varphi }\eta (b,a)].\ $Let $\left\vert f^{^{\prime
}}\right\vert ^{p}$ be $\varphi -$preinvex on $I,\ \frac{1}{p}+\frac{1}{q}%
=1\ $and $\eta (b,a)>0.\ $Then%
\begin{equation}
\begin{array}{l}
\left\vert f(a)+f(a+e^{i\varphi }\eta (b,a))-\dfrac{\Gamma (\alpha +1)}{%
\left[ e^{i\varphi }\eta (b,a)\right] ^{\alpha }}\left[ J_{a^{+}}^{\alpha
}f(x)+J_{\left( a+e^{i\varphi }\eta (b,a)\right) ^{-}}^{\alpha }f(x)\right]
\right\vert \\ 
\leq \left( \dfrac{2}{\alpha p+1}\right) ^{\frac{1}{p}}e^{i\varphi }\eta
(b,a)\left[ 1-\dfrac{1}{2^{\alpha p}}\right] ^{\dfrac{1}{p}}\left\{ \dfrac{%
\left\vert f^{^{\prime }}(a)\right\vert ^{\frac{p}{p-1}}+\left\vert
f^{^{\prime }}(b)\right\vert ^{\frac{p}{p-1}}}{2}\right\} ^{\frac{p-1}{p}}.%
\end{array}
\tag{3.7}
\end{equation}
\end{corollary}

\textbf{II. }If $h(t)=t^{s}$, then we have the result for $s_{\varphi }-$%
preinvexity

\begin{corollary}
Let $I\subseteq 
\mathbb{R}
$ be a open invex set with respect to $\eta :I\times I\rightarrow 
\mathbb{R}
$. Suppose that $f:I\rightarrow 
\mathbb{R}
$ is a differentiable function such that $f^{^{\prime }}\in
L_{1}[a,a+e^{i\varphi }\eta (b,a)]$. Let $\left\vert f^{^{\prime
}}\right\vert ^{p}$ be $s_{\varphi }-$preinvex on $I$, $\frac{1}{p}+\frac{1}{%
q}=1,$ $\eta (b,a)>0$ and $s\in (0,1].\ $Then%
\begin{equation}
\begin{array}{l}
\left\vert f(a)+f(a+e^{i\varphi }\eta (b,a))-\dfrac{\Gamma (\alpha +1)}{%
\left[ e^{i\varphi }\eta (b,a)\right] ^{\alpha }}\left[ J_{a^{+}}^{\alpha
}f(x)+J_{\left( a+e^{i\varphi }\eta (b,a)\right) ^{-}}^{\alpha }f(x)\right]
\right\vert \\ 
\leq \left( \dfrac{2}{\alpha p+1}\right) ^{\frac{1}{p}}e^{i\varphi }\eta
(b,a)\left[ 1-\dfrac{1}{2^{\alpha p}}\right] ^{\dfrac{1}{p}}\left\{ \dfrac{%
\left\vert f^{^{\prime }}(a)\right\vert ^{\frac{p}{p-1}}+\left\vert
f^{^{\prime }}(b)\right\vert ^{\frac{p}{p-1}}}{s+1}\right\} ^{\frac{p-1}{p}}.%
\end{array}
\tag{3.8}
\end{equation}
\end{corollary}

\textbf{III. }If $h(t)=1$, then we have the result for $P_{\varphi }-$%
preinvexity.

\begin{corollary}
Let $I\subseteq 
\mathbb{R}
$ be a open invex set with respect to $\eta :I\times I\rightarrow 
\mathbb{R}
$. Suppose that $f:I\rightarrow 
\mathbb{R}
$ is a differentiable function such that $f^{^{\prime }}\in
L_{1}[a,a+e^{i\varphi }\eta (b,a)]$. Let $\left\vert f^{^{\prime
}}\right\vert ^{p}$ be $P_{\varphi }-$preinvex on $I$, $\frac{1}{p}+\frac{1}{%
q}=1\ $and $\eta (b,a)>0.\ $Then%
\begin{equation}
\begin{array}{l}
\left\vert f(a)+f(a+e^{i\varphi }\eta (b,a))-\dfrac{\Gamma (\alpha +1)}{%
\left[ e^{i\varphi }\eta (b,a)\right] ^{\alpha }}\left[ J_{a^{+}}^{\alpha
}f(x)+J_{\left( a+e^{i\varphi }\eta (b,a)\right) ^{-}}^{\alpha }f(x)\right]
\right\vert \\ 
\leq \left( \dfrac{2}{\alpha p+1}\right) ^{\frac{1}{p}}e^{i\varphi }\eta
(b,a)\left[ 1-\dfrac{1}{2^{\alpha p}}\right] ^{\dfrac{1}{p}}\left\{
\left\vert f^{^{\prime }}(a)\right\vert ^{\frac{p}{p-1}}+\left\vert
f^{^{\prime }}(b)\right\vert ^{\frac{p}{p-1}}\right\} ^{\frac{p-1}{p}}.%
\end{array}
\tag{3.9}
\end{equation}
\end{corollary}

\begin{theorem}
Let $I\subseteq 
\mathbb{R}
$ be a open invex set with respect to $\eta :I\times I\rightarrow 
\mathbb{R}
$. Suppose that $f:I\rightarrow 
\mathbb{R}
$ is a differentiable function such that $f^{^{\prime }}\in
L_{1}[a,a+e^{i\varphi }\eta (b,a)]$. Let $\left\vert f^{^{\prime
}}\right\vert ^{q}$, $q>1\ $is $h_{\varphi }-$preinvex on $I$, then, for $%
\eta (b,a)>0$,%
\begin{equation}
\begin{array}{l}
\left\vert f(a)+f(a+e^{i\varphi }\eta (b,a))-\dfrac{\Gamma (\alpha +1)}{%
\left[ e^{i\varphi }\eta (b,a)\right] ^{\alpha }}\left[ J_{a^{+}}^{\alpha
}f(x)+J_{\left( a+e^{i\varphi }\eta (b,a)\right) ^{-}}^{\alpha }f(x)\right]
\right\vert \\ 
\leq \left( \dfrac{2}{\alpha +1}\right) ^{\frac{1}{p}}\left[ 1-\dfrac{1}{%
2^{\alpha }}\right] ^{\frac{1}{p}}e^{i\varphi }\eta (b,a)\left\{ \left[
\left\vert f^{^{\prime }}(a)\right\vert ^{\frac{p}{p-1}}+\left\vert
f^{^{\prime }}(b)\right\vert ^{\frac{p}{p-1}}\right] \dint\limits_{0}^{1}%
\left\vert \left( 1-t\right) ^{\alpha }-t^{\alpha }\right\vert
h(t)dt\right\} ^{\frac{p-1}{p}}.%
\end{array}
\tag{3.10}
\end{equation}
\end{theorem}

\begin{proof}
Using Lemma 2, we have%
\begin{equation*}
\begin{array}{l}
\left\vert f(a)+f(a+e^{i\varphi }\eta (b,a))-\dfrac{\Gamma (\alpha +1)}{%
\left[ e^{i\varphi }\eta (b,a)\right] ^{\alpha }}\left[ J_{a^{+}}^{\alpha
}f(x)+J_{\left( a+e^{i\varphi }\eta (b,a)\right) ^{-}}^{\alpha }f(x)\right]
\right\vert \\ 
=\left\vert e^{i\varphi }\eta (b,a)\dint\limits_{0}^{1}\left[ \left(
1-t\right) ^{\alpha }-t^{\alpha }\right] f^{^{\prime }}(a+te^{i\varphi }\eta
(b,a))dt\right\vert \\ 
\leq e^{i\varphi }\eta (b,a)\dint\limits_{0}^{1}\left\vert \left( 1-t\right)
^{\alpha }-t^{\alpha }\right\vert \left\vert f^{^{\prime }}(a+te^{i\varphi
}\eta (b,a))\right\vert dt%
\end{array}%
\end{equation*}%
Using power-mean inequality, we have%
\begin{equation*}
\begin{array}{l}
\left\vert f(a)+f(a+e^{i\varphi }\eta (b,a))-\dfrac{\Gamma (\alpha +1)}{%
\left[ e^{i\varphi }\eta (b,a)\right] ^{\alpha }}\left[ J_{a^{+}}^{\alpha
}f(x)+J_{\left( a+e^{i\varphi }\eta (b,a)\right) ^{-}}^{\alpha }f(x)\right]
\right\vert \\ 
\leq e^{i\varphi }\eta (b,a)\left( \dint\limits_{0}^{1}\left\vert \left(
1-t\right) ^{\alpha }-t^{\alpha }\right\vert dt\right) ^{\frac{1}{p}}\left(
\dint\limits_{0}^{1}\left\vert f^{^{\prime }}(a+te^{i\varphi }\eta
(b,a))\right\vert ^{q}dt\right) ^{\frac{1}{q}} \\ 
=e^{i\varphi }\eta (b,a)\left( \dint\limits_{0}^{\frac{1}{2}}\left( \left(
1-t\right) ^{\alpha }-t^{\alpha }\right) dt+\dint\limits_{\frac{1}{2}%
}^{1}\left( t^{\alpha }-\left( 1-t\right) ^{\alpha }\right) dt\right) ^{%
\frac{1}{p}} \\ 
\ \ \ \ \ \ \ \ \ \ \ \ \ \ \ \ \ \ \ \ \times \left(
\dint\limits_{0}^{1}\left\vert f^{^{\prime }}(a+te^{i\varphi }\eta
(b,a))\right\vert ^{q}dt\right) ^{\frac{1}{q}}%
\end{array}%
\end{equation*}%
now using $h_{\varphi }-$preinvexity of $\left\vert f^{^{\prime
}}\right\vert ^{q}$, we have%
\begin{equation*}
\begin{array}{l}
\left\vert f(a)+f(a+e^{i\varphi }\eta (b,a))-\dfrac{\Gamma (\alpha +1)}{%
\left[ e^{i\varphi }\eta (b,a)\right] ^{\alpha }}\left[ J_{a^{+}}^{\alpha
}f(x)+J_{\left( a+e^{i\varphi }\eta (b,a)\right) ^{-}}^{\alpha }f(x)\right]
\right\vert \\ 
\leq \dfrac{2^{\frac{1}{p}}}{\left( \alpha +1\right) ^{\frac{1}{p}}}%
e^{i\varphi }\eta (b,a)\left[ 1-\dfrac{1}{2^{\alpha }}\right] ^{\dfrac{1}{p}%
}\left( \dint\limits_{0}^{1}\left[ h(1-t)\left\vert f^{^{\prime
}}(a)\right\vert ^{q}+h(t)\left\vert f^{^{\prime }}(b)\right\vert ^{q}\right]
dt\right) ^{\frac{1}{q}} \\ 
=\left( \dfrac{2}{\alpha +1}\right) ^{\frac{1}{p}}e^{i\varphi }\eta (b,a)%
\left[ 1-\dfrac{1}{2^{\alpha }}\right] ^{\frac{1}{p}}\left\{ \left[
\left\vert f^{^{\prime }}(a)\right\vert ^{\frac{p}{p-1}}+\left\vert
f^{^{\prime }}(b)\right\vert ^{\frac{p}{p-1}}\right] \dint\limits_{0}^{1}%
\left\vert \left( 1-t\right) ^{\alpha }-t^{\alpha }\right\vert
h(t)dt\right\} ^{\frac{p-1}{p}}%
\end{array}%
\end{equation*}%
This completes the proof.
\end{proof}

Now we have some special cases for $\left( 3.10\right) $.

\textbf{I.} If $h(t)=t$, then we have the result for $\varphi -$preinvexity.

\begin{corollary}
Let $I\subseteq 
\mathbb{R}
$ be a open invex set with respect to $\eta :I\times I\rightarrow 
\mathbb{R}
.\ $Suppose that $f:I\rightarrow 
\mathbb{R}
$ is a differentiable function such that $f^{^{\prime }}\in
L_{1}[a,a+e^{i\varphi }\eta (b,a)].\ $If $\left\vert f^{^{\prime
}}\right\vert ^{q},\ q>1$ is $\varphi -$preinvex on $I,\ $then, for $\eta
(b,a)>0,$%
\begin{equation}
\begin{array}{l}
\left\vert f(a)+f(a+e^{i\varphi }\eta (b,a))-\dfrac{\Gamma (\alpha +1)}{%
\left[ e^{i\varphi }\eta (b,a)\right] ^{\alpha }}\left[ J_{a^{+}}^{\alpha
}f(x)+J_{\left( a+e^{i\varphi }\eta (b,a)\right) ^{-}}^{\alpha }f(x)\right]
\right\vert \\ 
\leq \left( \dfrac{2}{\alpha +1}\right) ^{\frac{1}{p}}\left[ 1-\dfrac{1}{%
2^{\alpha }}\right] ^{\frac{1}{p}}e^{i\varphi }\eta (b,a)\left\{ \left[
\left\vert f^{^{\prime }}(a)\right\vert ^{\frac{p}{p-1}}+\left\vert
f^{^{\prime }}(b)\right\vert ^{\frac{p}{p-1}}\right] \dfrac{1}{\alpha +2}%
\left[ 1-\dfrac{1}{2^{\alpha +2}}\right] \right\} ^{\frac{p-1}{p}}.%
\end{array}
\tag{3.11}
\end{equation}
\end{corollary}

\textbf{II. }If $h(t)=t^{s}$, then we have the result for $s_{\varphi }-$%
preinvexity

\begin{corollary}
Let $I\subseteq 
\mathbb{R}
$ be a open invex set with respect to $\eta :I\times I\rightarrow 
\mathbb{R}
$. Suppose that $f:I\rightarrow 
\mathbb{R}
$ is a differentiable function such that $f^{^{\prime }}\in
L_{1}[a,a+e^{i\varphi }\eta (b,a)]$. If $\left\vert f^{^{\prime
}}\right\vert ^{q},\ q>1$ is $s_{\varphi }-$preinvex on $I$, then, for $\eta
(b,a)>0$ and $s\in (0,1],$%
\begin{equation}
\begin{array}{l}
\left\vert f(a)+f(a+e^{i\varphi }\eta (b,a))-\dfrac{\Gamma (\alpha +1)}{%
\left[ e^{i\varphi }\eta (b,a)\right] ^{\alpha }}\left[ J_{a^{+}}^{\alpha
}f(x)+J_{\left( a+e^{i\varphi }\eta (b,a)\right) ^{-}}^{\alpha }f(x)\right]
\right\vert \\ 
\leq \left( \dfrac{2^{\alpha +1}-2}{2^{\alpha }\left( \alpha +1\right) }%
\right) e^{i\varphi }\eta (b,a)\left\{ \left[ \left\vert f^{^{\prime
}}(a)\right\vert ^{\frac{p}{p-1}}+\left\vert f^{^{\prime }}(b)\right\vert ^{%
\frac{p}{p-1}}\right] \right. \\ 
\left. x\left[ B_{\frac{1}{2}}\left( \alpha +1,s+1\right) -B_{\frac{1}{2}%
}\left( s+1,\alpha +1\right) +\dfrac{1}{\alpha +s+1}\left[ 1-\dfrac{1}{%
2^{s+\alpha }}\right] \right] \right\} ^{\frac{p-1}{p}}.%
\end{array}
\tag{3.12}
\end{equation}
\end{corollary}

\textbf{III. }If $h(t)=1$, then we have the result for $P_{\varphi }-$%
preinvexity.

\begin{corollary}
Let $I\subseteq 
\mathbb{R}
$ be a open invex set with respect to $\eta :I\times I\rightarrow 
\mathbb{R}
$. Suppose that $f:I\rightarrow 
\mathbb{R}
$ is a differentiable function such that $f^{^{\prime }}\in
L_{1}[a,a+e^{i\varphi }\eta (b,a)]$. If $\left\vert f^{^{\prime
}}\right\vert ^{q},\ q>1$ is $P_{\varphi }-$preinvex on $I$, then, for $\eta
(b,a)>0$%
\begin{equation}
\begin{array}{l}
\left\vert f(a)+f(a+e^{i\varphi }\eta (b,a))-\dfrac{\Gamma (\alpha +1)}{%
\left[ e^{i\varphi }\eta (b,a)\right] ^{\alpha }}\left[ J_{a^{+}}^{\alpha
}f(x)+J_{\left( a+e^{i\varphi }\eta (b,a)\right) ^{-}}^{\alpha }f(x)\right]
\right\vert \\ 
\leq \left( \dfrac{2^{\alpha +1}-2}{2^{\alpha }\left( \alpha +1\right) }%
\right) e^{i\varphi }\eta (b,a)\left\{ \left\vert f^{^{\prime
}}(a)\right\vert ^{\frac{p}{p-1}}+\left\vert f^{^{\prime }}(b)\right\vert ^{%
\frac{p}{p-1}}\right\} ^{\frac{p-1}{p}}.%
\end{array}
\tag{3.13}
\end{equation}
\end{corollary}

Using the technique of $[1]$, we prove the following result which helps us
in proving our next results.

\begin{lemma}
Let $I\subseteq 
\mathbb{R}
$ be a open invex set with respect to $\eta :I\times I\rightarrow 
\mathbb{R}
$ where $\eta (b,a)>0$. If $f^{^{\prime \prime }}\in L_{1}[a,a+e^{i\varphi
}\eta (b,a)]$, then%
\begin{equation}
\begin{array}{l}
f(a)+f(a+e^{i\varphi }\eta (b,a))-\dfrac{\Gamma (\alpha +1)}{\left[
e^{i\varphi }\eta (b,a)\right] ^{\alpha }}\left\{ J_{a^{+}}^{\alpha
}f(x)+J_{\left( a+e^{i\varphi }\eta (b,a)\right) ^{-}}^{\alpha }f(x)\right\}
\\ 
=\left[ e^{i\varphi }\eta (b,a)\right] ^{2}\dint\limits_{0}^{1}\left[ \dfrac{%
1-\left( 1-t\right) ^{\alpha +1}-t^{\alpha +1}}{\alpha +1}\right]
f^{^{\prime \prime }}(a+te^{i\varphi }\eta (b,a))dt.%
\end{array}
\tag{3.14}
\end{equation}
\end{lemma}

\begin{proof}
Let%
\begin{equation*}
\begin{array}{l}
\dint\limits_{0}^{1}\left[ \dfrac{1-\left( 1-t\right) ^{\alpha +1}-t^{\alpha
+1}}{\alpha +1}\right] f^{^{\prime \prime }}(a+te^{i\varphi }\eta (b,a))dt
\\ 
=\left. \dfrac{1-\left( 1-t\right) ^{\alpha +1}-t^{\alpha +1}}{\alpha +1}%
\dfrac{f^{^{\prime }}(a+te^{i\varphi }\eta (b,a))}{e^{i\varphi }\eta (b,a)}%
\right\vert _{0}^{1}-\dfrac{1}{e^{i\varphi }\eta (b,a)}\dint\limits_{0}^{1}%
\left[ \left( 1-t\right) ^{\alpha }-t^{\alpha }\right] f^{^{\prime
}}(a+te^{i\varphi }\eta (b,a))dt \\ 
=-\dfrac{1}{e^{i\varphi }\eta (b,a)}\dint\limits_{0}^{1}\left[ \left(
1-t\right) ^{\alpha }-t^{\alpha }\right] f^{^{\prime }}(a+te^{i\varphi }\eta
(b,a))dt \\ 
=\dfrac{1}{\left( e^{i\varphi }\eta (b,a)\right) ^{2}}\left[ \left[
f(a)+f(a+e^{i\varphi }\eta (b,a))\right] -\dfrac{\Gamma (\alpha +1)}{\left[
e^{i\varphi }\eta (b,a)\right] ^{\alpha }}\left\{ J_{a^{+}}^{\alpha
}f(x)+J_{\left( a+e^{i\varphi }\eta (b,a)\right) ^{-}}^{\alpha }f(x)\right\} %
\right]%
\end{array}%
\end{equation*}%
This completes the proof.
\end{proof}

\begin{theorem}
Let $I\subseteq 
\mathbb{R}
$ be a open invex set with respect to $\eta :I\times I\rightarrow 
\mathbb{R}
.$ Let $f:I\rightarrow 
\mathbb{R}
$ be a twice differentiable function such that $f^{^{\prime \prime }}\in
L_{1}[a,a+e^{i\varphi }\eta (b,a)]$. If $\left\vert f^{^{\prime \prime
}}\right\vert ,\ $is $h_{\varphi }-$preinvex on $I$, then, for $\eta (b,a)>0$%
,%
\begin{equation}
\begin{array}{l}
\left\vert f(a)+f(a+e^{i\varphi }\eta (b,a))-\dfrac{\Gamma (\alpha +1)}{%
\left[ e^{i\varphi }\eta (b,a)\right] ^{\alpha }}\left\{ J_{a^{+}}^{\alpha
}f(x)+J_{\left( a+e^{i\varphi }\eta (b,a)\right) ^{-}}^{\alpha }f(x)\right\}
\right\vert \\ 
\leq \left[ e^{i\varphi }\eta (b,a)\right] ^{2}\left\{ \left( \left\vert
f^{^{\prime \prime }}(a)\right\vert +\left\vert f^{^{\prime \prime
}}(b)\right\vert \right) \dint\limits_{0}^{1}\left[ \dfrac{1-\left(
1-t\right) ^{\alpha +1}-t^{\alpha +1}}{\alpha +1}\right] h(t)dt\right\} .%
\end{array}
\tag{3.15}
\end{equation}
\end{theorem}

\begin{proof}
Using Lemma 3, we have%
\begin{equation*}
\begin{array}{l}
\left\vert f(a)+f(a+e^{i\varphi }\eta (b,a))-\dfrac{\Gamma (\alpha +1)}{%
\left[ e^{i\varphi }\eta (b,a)\right] ^{\alpha }}\left\{ J_{a^{+}}^{\alpha
}f(x)+J_{\left( a+e^{i\varphi }\eta (b,a)\right) ^{-}}^{\alpha }f(x)\right\}
\right\vert \\ 
=\left\vert \left[ e^{i\varphi }\eta (b,a)\right] ^{2}\dint\limits_{0}^{1}%
\left[ \dfrac{1-\left( 1-t\right) ^{\alpha +1}-t^{\alpha +1}}{\alpha +1}%
\right] f^{^{\prime \prime }}(a+te^{i\varphi }\eta (b,a))dt\right\vert \\ 
\leq \left[ e^{i\varphi }\eta (b,a)\right] ^{2}\dint\limits_{0}^{1}\left[ 
\dfrac{1-\left( 1-t\right) ^{\alpha +1}-t^{\alpha +1}}{\alpha +1}\right]
\left\vert f^{^{\prime \prime }}(a+te^{i\varphi }\eta (b,a))\right\vert dt
\\ 
\leq \left[ e^{i\varphi }\eta (b,a)\right] ^{2}\dint\limits_{0}^{1}\left[ 
\dfrac{1-\left( 1-t\right) ^{\alpha +1}-t^{\alpha +1}}{\alpha +1}\right]
\left( h(1-t)\left\vert f^{^{\prime \prime }}(a)\right\vert +h(t)\left\vert
f^{^{\prime \prime }}(b)\right\vert \right) dt \\ 
=\left[ e^{i\varphi }\eta (b,a)\right] ^{2}\left( \left\vert f^{^{\prime
\prime }}(a)\right\vert +\left\vert f^{^{\prime \prime }}(b)\right\vert
\right) \dint\limits_{0}^{1}\left[ \dfrac{1-\left( 1-t\right) ^{\alpha
+1}-t^{\alpha +1}}{\alpha +1}\right] h(t)dt%
\end{array}%
\end{equation*}%
This completes the proof.
\end{proof}

Now, we discuss some special cases for $\left( 3.15\right) $.

\textbf{I.} If $h(t)=t$, then we have the result for $\varphi -$preinvexity.

\begin{corollary}
Let $I\subseteq 
\mathbb{R}
$ be a open invex set with respect to $\eta :I\times I\rightarrow 
\mathbb{R}
.\ $Suppose that $f:I\rightarrow 
\mathbb{R}
$ be a twice differentiable function such that $f^{^{\prime \prime }}\in
L_{1}[a,a+e^{i\varphi }\eta (b,a)].\ $If $\left\vert f^{^{\prime \prime
}}\right\vert $is $\varphi -$preinvex on $I,\ $then, for $\eta (b,a)>0,$%
\begin{equation}
\begin{array}{l}
\left\vert f(a)+f(a+e^{i\varphi }\eta (b,a))-\dfrac{\Gamma (\alpha +1)}{%
\left[ e^{i\varphi }\eta (b,a)\right] ^{\alpha }}\left\{ J_{a^{+}}^{\alpha
}f(x)+J_{\left( a+e^{i\varphi }\eta (b,a)\right) ^{-}}^{\alpha }f(x)\right\}
\right\vert \\ 
\leq \left[ e^{i\varphi }\eta (b,a)\right] ^{2}\left\{ \dfrac{\alpha }{%
2\left( \alpha +1\right) \left( \alpha +2\right) }\left( \left\vert
f^{^{\prime \prime }}(a)\right\vert +\left\vert f^{^{\prime \prime
}}(b)\right\vert \right) \right\} .%
\end{array}
\tag{3.16}
\end{equation}
\end{corollary}

\textbf{II. }If $h(t)=t^{s}$, then we have the result for $s_{\varphi }-$%
preinvexity

\begin{corollary}
Let $I\subseteq 
\mathbb{R}
$ be a open invex set with respect to $\eta :I\times I\rightarrow 
\mathbb{R}
$. Let $f:I\rightarrow 
\mathbb{R}
$ be a twice differentiable function such that $f^{^{\prime \prime }}\in
L_{1}[a,a+e^{i\varphi }\eta (b,a)]$. If $\left\vert f^{^{\prime \prime
}}\right\vert $ is $s_{\varphi }-$preinvex on $I$, then, for $\eta (b,a)>0$
and $s\in (0,1]$,%
\begin{equation}
\begin{array}{l}
\left\vert f(a)+f(a+e^{i\varphi }\eta (b,a))-\dfrac{\Gamma (\alpha +1)}{%
\left[ e^{i\varphi }\eta (b,a)\right] ^{\alpha }}\left\{ J_{a^{+}}^{\alpha
}f(x)+J_{\left( a+e^{i\varphi }\eta (b,a)\right) ^{-}}^{\alpha }f(x)\right\}
\right\vert \\ 
\leq \left[ e^{i\varphi }\eta (b,a)\right] ^{2}\left\{ \left( \left\vert
f^{^{\prime \prime }}(a)\right\vert +\left\vert f^{^{\prime \prime
}}(b)\right\vert \right) \left( \dfrac{1}{\left( \alpha +s+2\right) \left(
s+1\right) }-\dfrac{B_{\frac{1}{2}}\left( s+1,\alpha \right) }{\alpha +1}%
\right) \right\} .%
\end{array}
\tag{3.17}
\end{equation}
\end{corollary}

\textbf{III. }If $h(t)=1$, then we have the result for $P_{\varphi }-$%
preinvexity.

\begin{corollary}
Let $I\subseteq 
\mathbb{R}
$ be a open invex set with respect to $\eta :I\times I\rightarrow 
\mathbb{R}
$. Let $f:I\rightarrow 
\mathbb{R}
$ be a twice differentiable function such that $f^{^{\prime \prime }}\in
L_{1}[a,a+e^{i\varphi }\eta (b,a)]$. If $\left\vert f^{^{\prime \prime
}}\right\vert $ is $P_{\varphi }-$preinvex on $I$, then, for $\eta (b,a)>0$,%
\begin{equation}
\begin{array}{l}
\left\vert f(a)+f(a+e^{i\varphi }\eta (b,a))-\dfrac{\Gamma (\alpha +1)}{%
\left[ e^{i\varphi }\eta (b,a)\right] ^{\alpha }}\left\{ J_{a^{+}}^{\alpha
}f(x)+J_{\left( a+e^{i\varphi }\eta (b,a)\right) ^{-}}^{\alpha }f(x)\right\}
\right\vert \\ 
\leq \left[ e^{i\varphi }\eta (b,a)\right] ^{2}\left\{ \left( \left\vert
f^{^{\prime \prime }}(a)\right\vert +\left\vert f^{^{\prime \prime
}}(b)\right\vert \right) \dfrac{\alpha }{\left( \alpha +1\right) \left(
\alpha +2\right) }\right\} .%
\end{array}
\tag{3.18}
\end{equation}
\end{corollary}

\begin{theorem}
Let $I\subseteq 
\mathbb{R}
$ be a open invex set with respect to $\eta :I\times I\rightarrow 
\mathbb{R}
$. Let $f:I\rightarrow 
\mathbb{R}
$ be a twice differentiable function such that $f^{^{\prime \prime }}\in
L_{1}[a,a+e^{i\varphi }\eta (b,a)]$. If $\left\vert f^{^{\prime \prime
}}\right\vert $ is $P_{\varphi }-$preinvex on $I$, then, for $\eta (b,a)>0$,%
\begin{equation}
\begin{array}{l}
\left\vert f(a)+f(a+e^{i\varphi }\eta (b,a))-\dfrac{\Gamma (\alpha +1)}{%
\left[ e^{i\varphi }\eta (b,a)\right] ^{\alpha }}\left\{ J_{a^{+}}^{\alpha
}f(x)+J_{\left( a+e^{i\varphi }\eta (b,a)\right) ^{-}}^{\alpha }f(x)\right\}
\right\vert \\ 
\leq \left[ e^{i\varphi }\eta (b,a)\right] ^{2}\left( 1-\dfrac{1}{2^{\alpha }%
}\right) \left[ \left\vert f^{^{\prime \prime }}(a)\right\vert
^{q}+\left\vert f^{^{\prime \prime }}(b)\right\vert ^{q}\right] ^{\frac{1}{q}%
}\left( \dint\limits_{0}^{1}h(t)dt\right) ^{\frac{1}{q}}%
\end{array}
\tag{3.19}
\end{equation}
\end{theorem}

\begin{proof}
Using Lemma 3, we have%
\begin{equation*}
\begin{array}{l}
\left\vert f(a)+f(a+e^{i\varphi }\eta (b,a))-\dfrac{\Gamma (\alpha +1)}{%
\left[ e^{i\varphi }\eta (b,a)\right] ^{\alpha }}\left\{ J_{a^{+}}^{\alpha
}f(x)+J_{\left( a+e^{i\varphi }\eta (b,a)\right) ^{-}}^{\alpha }f(x)\right\}
\right\vert \\ 
=\left\vert \left[ e^{i\varphi }\eta (b,a)\right] ^{2}\dint\limits_{0}^{1}%
\left[ \dfrac{1-\left( 1-t\right) ^{\alpha +1}-t^{\alpha +1}}{\alpha +1}%
\right] f^{^{\prime \prime }}(a+te^{i\varphi }\eta (b,a))dt\right\vert \\ 
\leq \left[ e^{i\varphi }\eta (b,a)\right] ^{2}\left( \dint\limits_{0}^{1}%
\left[ \dfrac{1-\left( 1-t\right) ^{\alpha +1}-t^{\alpha +1}}{\alpha +1}%
\right] ^{p}dt\right) ^{\frac{1}{p}}\left( \dint\limits_{0}^{1}\left\vert
f^{^{\prime \prime }}(a+te^{i\varphi }\eta (b,a))\right\vert ^{q}dt\right) ^{%
\frac{1}{q}} \\ 
\leq \dfrac{\left[ e^{i\varphi }\eta (b,a)\right] ^{2}}{\alpha +1}\left( 1-%
\dfrac{1}{2^{\alpha }}\right) \left( \dint\limits_{0}^{1}\left\{
h(1-t)\left\vert f^{^{\prime \prime }}(a)\right\vert ^{q}+h(t)\left\vert
f^{^{\prime \prime }}(b)\right\vert ^{q}\right\} dt\right) ^{\frac{1}{q}} \\ 
=\dfrac{\left[ e^{i\varphi }\eta (b,a)\right] ^{2}}{\alpha +1}\left( 1-%
\dfrac{1}{2^{\alpha }}\right) \left( \left\{ \left\vert f^{^{\prime \prime
}}(a)\right\vert ^{q}+\left\vert f^{^{\prime \prime }}(b)\right\vert
^{q}\right\} \dint\limits_{0}^{1}\left[ h(t)\right] dt\right) ^{\frac{1}{q}}%
\end{array}%
\end{equation*}%
This completes the proof.
\end{proof}

We have some special cases for $\left( 3.19\right) $.

\textbf{I.} If $h(t)=t$, then we have the result for $\varphi -$preinvexity.

\begin{corollary}
Let $I\subseteq 
\mathbb{R}
$ be a open invex set with respect to $\eta :I\times I\rightarrow 
\mathbb{R}
.\ $Suppose that $f:I\rightarrow 
\mathbb{R}
$ be a twice differentiable function such that $f^{^{\prime \prime }}\in
L_{1}[a,a+e^{i\varphi }\eta (b,a)].\ $If $\left\vert f^{^{\prime \prime
}}\right\vert $is $\varphi -$preinvex on $I,\ $then, for $\eta (b,a)>0,$%
\begin{equation}
\begin{array}{l}
\left\vert f(a)+f(a+e^{i\varphi }\eta (b,a))-\dfrac{\Gamma (\alpha +1)}{%
\left[ e^{i\varphi }\eta (b,a)\right] ^{\alpha }}\left\{ J_{a^{+}}^{\alpha
}f(x)+J_{\left( a+e^{i\varphi }\eta (b,a)\right) ^{-}}^{\alpha }f(x)\right\}
\right\vert \\ 
\leq \dfrac{\left[ e^{i\varphi }\eta (b,a)\right] ^{2}}{\alpha +1}\left( 1-%
\dfrac{1}{2^{\alpha }}\right) \left( \dfrac{1}{2}\right) ^{\frac{1}{q}%
}\left( \left\vert f^{^{\prime \prime }}(a)\right\vert ^{q}+\left\vert
f^{^{\prime \prime }}(b)\right\vert ^{q}\right) ^{\frac{1}{q}}%
\end{array}
\tag{3.20}
\end{equation}
\end{corollary}

\textbf{II. }If $h(t)=t^{s}$, then we have the result for $s_{\varphi }-$%
preinvexity

\begin{corollary}
Let $I\subseteq 
\mathbb{R}
$ be a open invex set with respect to $\eta :I\times I\rightarrow 
\mathbb{R}
$. Let $f:I\rightarrow 
\mathbb{R}
$ be a twice differentiable function such that $f^{^{\prime \prime }}\in
L_{1}[a,a+e^{i\varphi }\eta (b,a)]$. If $\left\vert f^{^{\prime \prime
}}\right\vert $ is $s_{\varphi }-$preinvex on $I$, then, for $\eta (b,a)>0$
and $s\in (0,1]$,%
\begin{equation}
\begin{array}{l}
\left\vert f(a)+f(a+e^{i\varphi }\eta (b,a))-\dfrac{\Gamma (\alpha +1)}{%
\left[ e^{i\varphi }\eta (b,a)\right] ^{\alpha }}\left\{ J_{a^{+}}^{\alpha
}f(x)+J_{\left( a+e^{i\varphi }\eta (b,a)\right) ^{-}}^{\alpha }f(x)\right\}
\right\vert \\ 
\leq \dfrac{\left[ e^{i\varphi }\eta (b,a)\right] ^{2}}{\alpha +1}\left( 1-%
\dfrac{1}{2^{\alpha }}\right) \left( \dfrac{1}{s+1}\right) ^{\frac{1}{q}%
}\left( \left\vert f^{^{\prime \prime }}(a)\right\vert ^{q}+\left\vert
f^{^{\prime \prime }}(b)\right\vert ^{q}\right) ^{\frac{1}{q}}%
\end{array}
\tag{3.21}
\end{equation}
\end{corollary}

\textbf{III. }If $h(t)=1$, then we have the result for $P_{\varphi }-$%
preinvexity.

\begin{corollary}
Let $I\subseteq 
\mathbb{R}
$ be a open invex set with respect to $\eta :I\times I\rightarrow 
\mathbb{R}
$. Let $f:I\rightarrow 
\mathbb{R}
$ be a twice differentiable function such that $f^{^{\prime \prime }}\in
L_{1}[a,a+e^{i\varphi }\eta (b,a)]$. If $\left\vert f^{^{\prime \prime
}}\right\vert $ is $P_{\varphi }-$preinvex on $I$, then, for $\eta (b,a)>0$,%
\begin{equation}
\begin{array}{l}
\left\vert f(a)+f(a+e^{i\varphi }\eta (b,a))-\dfrac{\Gamma (\alpha +1)}{%
\left[ e^{i\varphi }\eta (b,a)\right] ^{\alpha }}\left\{ J_{a^{+}}^{\alpha
}f(x)+J_{\left( a+e^{i\varphi }\eta (b,a)\right) ^{-}}^{\alpha }f(x)\right\}
\right\vert \\ 
\leq \dfrac{\left[ e^{i\varphi }\eta (b,a)\right] ^{2}}{\alpha +1}\left( 1-%
\dfrac{1}{2^{\alpha }}\right) \left( \left\vert f^{^{\prime \prime
}}(a)\right\vert ^{q}+\left\vert f^{^{\prime \prime }}(b)\right\vert
^{q}\right) ^{\frac{1}{q}}%
\end{array}
\tag{3.22}
\end{equation}
\end{corollary}

\begin{theorem}
Let $I\subseteq 
\mathbb{R}
$ be a open invex set with respect to $\eta :I\times I\rightarrow 
\mathbb{R}
$. Let $f:I\rightarrow 
\mathbb{R}
$ be a twice differentiable function such that $f^{^{\prime \prime }}\in
L_{1}[a,a+e^{i\varphi }\eta (b,a)]$. If $\left\vert f^{^{\prime \prime
}}\right\vert $ is $P_{\varphi }-$preinvex on $I$, then, for $\eta (b,a)>0$,%
\begin{equation}
\begin{array}{l}
\left\vert f(a)+f(a+e^{i\varphi }\eta (b,a))-\dfrac{\Gamma (\alpha +1)}{%
\left[ e^{i\varphi }\eta (b,a)\right] ^{\alpha }}\left\{ J_{a^{+}}^{\alpha
}f(x)+J_{\left( a+e^{i\varphi }\eta (b,a)\right) ^{-}}^{\alpha }f(x)\right\}
\right\vert \\ 
\leq \left[ e^{i\varphi }\eta (b,a)\right] ^{2}\left( \dfrac{\alpha }{\left(
\alpha +1\right) \left( \alpha +2\right) }\right) ^{\frac{1}{p}} \\ 
\times \left( \left\{ \left\vert f^{^{\prime \prime }}(a)\right\vert
^{q}+\left\vert f^{^{\prime \prime }}(b)\right\vert ^{q}\right\}
\dint\limits_{0}^{1}\left[ \dfrac{1-\left( 1-t\right) ^{\alpha +1}-t^{\alpha
+1}}{\alpha +1}\right] h(t)dt\right) ^{\frac{1}{q}}%
\end{array}
\tag{3.23}
\end{equation}
\end{theorem}

\begin{proof}
Using Lemma 3 and well-known power-mean inequality, we have%
\begin{equation*}
\begin{array}{l}
\left\vert f(a)+f(a+e^{i\varphi }\eta (b,a))-\dfrac{\Gamma (\alpha +1)}{%
\left[ e^{i\varphi }\eta (b,a)\right] ^{\alpha }}\left\{ J_{a^{+}}^{\alpha
}f(x)+J_{\left( a+e^{i\varphi }\eta (b,a)\right) ^{-}}^{\alpha }f(x)\right\}
\right\vert \\ 
=\left\vert \left[ e^{i\varphi }\eta (b,a)\right] ^{2}\dint\limits_{0}^{1}%
\left[ \dfrac{1-\left( 1-t\right) ^{\alpha +1}-t^{\alpha +1}}{\alpha +1}%
\right] f^{^{\prime \prime }}(a+te^{i\varphi }\eta (b,a))dt\right\vert \\ 
\leq \left[ e^{i\varphi }\eta (b,a)\right] ^{2}\left( \dint\limits_{0}^{1}%
\left[ \dfrac{1-\left( 1-t\right) ^{\alpha +1}-t^{\alpha +1}}{\alpha +1}%
\right] dt\right) ^{\frac{1}{p}} \\ 
\times \left( \dint\limits_{0}^{1}\left[ \dfrac{1-\left( 1-t\right) ^{\alpha
+1}-t^{\alpha +1}}{\alpha +1}\right] \left\vert f^{^{\prime \prime
}}(a+te^{i\varphi }\eta (b,a))\right\vert ^{q}dt\right) ^{\frac{1}{q}} \\ 
\leq \left[ e^{i\varphi }\eta (b,a)\right] ^{2}\left( \dint\limits_{0}^{1}%
\left[ \dfrac{1-\left( 1-t\right) ^{\alpha +1}-t^{\alpha +1}}{\alpha +1}%
\right] dt\right) ^{\frac{1}{p}} \\ 
\times \left( \dint\limits_{0}^{1}\left[ \dfrac{1-\left( 1-t\right) ^{\alpha
+1}-t^{\alpha +1}}{\alpha +1}\right] \left\{ h(1-t)\left\vert f^{^{\prime
\prime }}(a)\right\vert ^{q}+h(t)\left\vert f^{^{\prime \prime
}}(b)\right\vert ^{q}\right\} dt\right) ^{\frac{1}{q}} \\ 
\leq \left[ e^{i\varphi }\eta (b,a)\right] ^{2}\left( \dfrac{\alpha }{\left(
\alpha +1\right) \left( \alpha +2\right) }\right) ^{\frac{1}{p}}\left\{
\dint\limits_{0}^{1}\left[ \dfrac{1-\left( 1-t\right) ^{\alpha +1}-t^{\alpha
+1}}{\alpha +1}\right] \right. \\ 
\times \left. \left( h(1-t)\left\vert f^{^{\prime \prime }}(a)\right\vert
^{q}+h(t)\left\vert f^{^{\prime \prime }}(b)\right\vert ^{q}\right)
dt\right\} ^{\frac{1}{q}} \\ 
=\left[ e^{i\varphi }\eta (b,a)\right] ^{2}\left( \dfrac{\alpha }{\left(
\alpha +1\right) \left( \alpha +2\right) }\right) ^{\frac{1}{p}}\left\{
\left\{ \left\vert f^{^{\prime \prime }}(a)\right\vert ^{q}+\left\vert
f^{^{\prime \prime }}(b)\right\vert ^{q}\right\} \right. \\ 
\times \left. \dint\limits_{0}^{1}\left[ \dfrac{1-\left( 1-t\right) ^{\alpha
+1}-t^{\alpha +1}}{\alpha +1}\right] h(t)dt\right\} ^{\frac{1}{q}}%
\end{array}%
\end{equation*}%
This completes the proof.
\end{proof}

We have some special cases for $\left( 3.23\right) $.

\textbf{I.} If $h(t)=t$, then we have the result for $\varphi -$preinvexity.

\begin{corollary}
Let $I\subseteq 
\mathbb{R}
$ be a open invex set with respect to $\eta :I\times I\rightarrow 
\mathbb{R}
.\ $Suppose that $f:I\rightarrow 
\mathbb{R}
$ be a twice differentiable function such that $f^{^{\prime \prime }}\in
L_{1}[a,a+e^{i\varphi }\eta (b,a)].\ $If $\left\vert f^{^{\prime \prime
}}\right\vert $is $\varphi -$preinvex on $I,\ $then, for every $\eta
(b,a)>0, $%
\begin{equation}
\begin{array}{l}
\left\vert f(a)+f(a+e^{i\varphi }\eta (b,a))-\dfrac{\Gamma (\alpha +1)}{%
\left[ e^{i\varphi }\eta (b,a)\right] ^{\alpha }}\left\{ J_{a^{+}}^{\alpha
}f(x)+J_{\left( a+e^{i\varphi }\eta (b,a)\right) ^{-}}^{\alpha }f(x)\right\}
\right\vert \\ 
\leq \left[ e^{i\varphi }\eta (b,a)\right] ^{2}\left( \dfrac{1}{2}\right) ^{%
\frac{1}{q}}\left( \dfrac{\alpha }{\left( \alpha +1\right) \left( \alpha
+2\right) }\right) \left( \left\{ \left\vert f^{^{\prime \prime
}}(a)\right\vert ^{q}+\left\vert f^{^{\prime \prime }}(b)\right\vert
^{q}\right\} \right) ^{\frac{1}{q}}%
\end{array}
\tag{3.24}
\end{equation}
\end{corollary}

\textbf{II. }If $h(t)=t^{s}$, then we have the result for $s_{\varphi }-$%
preinvexity

\begin{corollary}
Let $I\subseteq 
\mathbb{R}
$ be a open invex set with respect to $\eta :I\times I\rightarrow 
\mathbb{R}
$. Let $f:I\rightarrow 
\mathbb{R}
$ be a twice differentiable function such that $f^{^{\prime \prime }}\in
L_{1}[a,a+e^{i\varphi }\eta (b,a)]$. If $\left\vert f^{^{\prime \prime
}}\right\vert $ is $s_{\varphi }-$preinvex on $I$, then, for every $\eta
(b,a)>0$ and $s\in (0,1]$,%
\begin{equation}
\begin{array}{l}
\left\vert f(a)+f(a+e^{i\varphi }\eta (b,a))-\dfrac{\Gamma (\alpha +1)}{%
\left[ e^{i\varphi }\eta (b,a)\right] ^{\alpha }}\left\{ J_{a^{+}}^{\alpha
}f(x)+J_{\left( a+e^{i\varphi }\eta (b,a)\right) ^{-}}^{\alpha }f(x)\right\}
\right\vert \\ 
\leq \left[ e^{i\varphi }\eta (b,a)\right] ^{2}\left( \dfrac{\alpha }{\left(
\alpha +1\right) \left( \alpha +2\right) }\right) ^{\frac{1}{p}} \\ 
\times \left( \left\{ \left\vert f^{^{\prime \prime }}(a)\right\vert
^{q}+\left\vert f^{^{\prime \prime }}(b)\right\vert ^{q}\right\} \left( 
\dfrac{1}{\left( \alpha +s+2\right) \left( s+1\right) }-\dfrac{B_{\frac{1}{2}%
}\left( s+1,\alpha \right) }{\alpha +1}\right) \right) ^{\frac{1}{q}}%
\end{array}
\tag{3.25}
\end{equation}
\end{corollary}

\textbf{III. }If $h(t)=1$, then we have the result for $P_{\varphi }-$%
preinvexity.

\begin{corollary}
Let $I\subseteq 
\mathbb{R}
$ be a open invex set with respect to $\eta :I\times I\rightarrow 
\mathbb{R}
$. Let $f:I\rightarrow 
\mathbb{R}
$ be a twice differentiable function such that $f^{^{\prime \prime }}\in
L_{1}[a,a+e^{i\varphi }\eta (b,a)]$. If $\left\vert f^{^{\prime \prime
}}\right\vert $ is $P_{\varphi }-$preinvex on $I$, then, for every $\eta
(b,a)>0$,%
\begin{equation}
\begin{array}{l}
\left\vert f(a)+f(a+e^{i\varphi }\eta (b,a))-\dfrac{\Gamma (\alpha +1)}{%
\left[ e^{i\varphi }\eta (b,a)\right] ^{\alpha }}\left\{ J_{a^{+}}^{\alpha
}f(x)+J_{\left( a+e^{i\varphi }\eta (b,a)\right) ^{-}}^{\alpha }f(x)\right\}
\right\vert \\ 
\leq \left[ e^{i\varphi }\eta (b,a)\right] ^{2}\left( \dfrac{\alpha }{\left(
\alpha +1\right) \left( \alpha +2\right) }\right) \left( \left\vert
f^{^{\prime \prime }}(a)\right\vert ^{q}+\left\vert f^{^{\prime \prime
}}(b)\right\vert ^{q}\right) ^{\frac{1}{q}}%
\end{array}
\tag{3.26}
\end{equation}
\end{corollary}

\begin{conclusion}
If we take $\alpha =1$ in $\left( 3.3\right) -\left( 3.5\right) $,$\ \left(
3.7\right) -\left( 3.9\right) $,$\ \left( 3.11\right) -\left( 3.13\right) $,$%
\ \left( 3.16\right) -\left( 3.18\right) $,$\ \left( 3.20\right) -\left(
3.22\right) $,$\ \left( 3.24\right) -\left( 3.26\right) $, we obtain
Crollary $3.3-3.26$ in $\left[ 20\right] $.
\end{conclusion}

\begin{conclusion}
If we take $\varphi =0\ $in $\left( 3.1\right) $, we obtain Lemma $2$.$3$ in 
$\left[ 4\right] $.
\end{conclusion}

\begin{conclusion}
If we take $\alpha =1\ $and $\varphi =0$ in our results, we get some
Hermite-Hadamard inequalities.\newpage
\end{conclusion}

\end{document}